\documentclass{amsart}
\usepackage{amssymb}
\usepackage{amsmath}
\usepackage{latexsym}
\usepackage{graphicx}
\usepackage{amscd}
\usepackage{mathrsfs}
\usepackage[english]{babel}
\usepackage{hyperref}
\usepackage{verbatim}
\usepackage{tikz-cd}
\usepackage{enumerate}
\usetikzlibrary{math}

\newtheorem{theorem}[equation]{Theorem}
\newtheorem{corollary}[equation]{Corollary}
\newtheorem{prop}[equation]{Proposition}
\newtheorem{lemma}[equation]{Lemma}
\newtheorem{cor}[equation]{Corollary}

\theoremstyle{definition}
\newtheorem{definition}[equation]{Definition}

\newtheorem{example}[equation]{Example}
\newtheorem{question}[equation]{Question}
\newtheorem{remark}[equation]{Remark}

\newtheorem*{theorem*}{Main Theorem}

\newcommand{\llbr}{[\negthinspace[}
\newcommand{\rrbr}{]\negthinspace]}

\newcommand{\llpa}{(\negthinspace(}
\newcommand{\rrpa}{)\negthinspace)}

\newcommand{\A}{\ensuremath{\mathbb{A}}}
\newcommand{\LL}{\ensuremath{\mathbb{L}}}

\newcommand{\Z}{\ensuremath{\mathbb{Z}}}
\newcommand{\PP}{\ensuremath{\mathbb{P}}}
\newcommand{\Q}{\ensuremath{\mathbb{Q}}}

\newcommand{\C}{\ensuremath{\mathbb{C}}}

\newcommand{\Pro}{\ensuremath{\mathbb{P}}}

\newcommand{\cU}{\ensuremath{\mathscr{U}}}
\newcommand{\cV}{\ensuremath{\mathscr{V}}}
\newcommand{\cW}{\ensuremath{\mathscr{W}}}
\newcommand{\cX}{\ensuremath{\mathscr{X}}}
\newcommand{\cY}{\ensuremath{\mathscr{Y}}}
\newcommand{\cZ}{\ensuremath{\mathscr{Z}}}

\newcommand{\Spec}{\ensuremath{\mathrm{Spec}\,}}

\newcommand{\red}{\mathrm{red}}

\newcommand{\Var}{\mathrm{Var}}

\newcommand{\Gro}{\mathbf{K}}

\newcommand{\gro}{\mathbf{K}}
\newcommand{\Res}{\mathrm{Res}}

\newcommand{\cR}{\ensuremath{\mathcal{R}}}

\numberwithin{equation}{subsection}

\author{Luigi Lunardon}
\address{Imperial College,
Department of Mathematics, South Kensington Campus,
London SW72AZ, UK.} \email{l.lunardon16@imperial.ac.uk}

\author{Johannes Nicaise}
\address{Imperial College,
Department of Mathematics, South Kensington Campus,
London SW72AZ, UK, and KU Leuven, Department of Mathematics, Celestijnenlaan 200B, 3001 Heverlee, Belgium.} \email{j.nicaise@imperial.ac.uk}

\begin{document}
\title[Birational invariance of motivic zeta functions]{Birational invariance of motivic zeta functions of $K$-trivial varieties, and obstructions to smooth fillings}

\begin{abstract}
The motivic zeta function of a smooth and proper $\C\llpa t\rrpa$-variety $X$ with trivial canonical bundle is a rational function with coefficients in an appropriate Grothendieck ring of complex varieties, which measures how $X$ degenerates at $t=0$. In analogy with Igusa's monodromy conjecture for $p$-adic zeta functions of hypersurface singularities, we expect that the poles of the motivic zeta function of $X$ correspond to monodromy eigenvalues on the cohomology of $X$. In this paper, we prove that the motivic zeta function and the monodromy conjecture are preserved under birational equivalence, which extends the range of known cases. As a further application, we explain how the motivic zeta function acts as an obstruction to the existence of a smooth filling for $X$ at $t=0$.
\end{abstract}

\maketitle

\section{Introduction}\label{sec:intro}
In 1995, Kontsevich presented the theory of motivic integration at a famous lecture in Orsay. His main motivation was to prove that birationally equivalent complex Calabi-Yau varieties have the same Hodge numbers. This refined the analogous result for Betti numbers that was obtained by Batyrev by means of $p$-adic integration techniques \cite{batyrev}; both  theorems were motivated by mirror symmetry. Kontsevich deduced the equality of Hodge numbers by proving a stronger result, namely, that the varieties in question have the same class in a suitable Grothendieck ring.

 The main result in this paper is a similar statement for the motivic zeta function of a smooth and proper $\C\llpa t\rrpa$-variety $X$ with trivial canonical bundle \cite{mallorca}. This motivic zeta function is a rational function with coefficients in an appropriate Grothendieck ring of complex varieties. It measures how $X$ degenerates at $t=0$ and is a natural analog of Denef and Loeser's motivic zeta function for hypersurface singularities \cite{DLbarc} {\em via} the non-archimedean interpretation of the Denef-Loeser zeta function in \cite{NiSe}.
 
 We prove in Corollary \ref{cor:birat} that the motivic zeta function is a birational invariant of $X$.
To do so, we set up a more general theory that associates motivic zeta functions with germs of canonical forms on smooth $\C\llpa t\rrpa$-schemes. The definition of this generalized motivic zeta function will be stated in terms of strict normal crossings (snc) models of the germ (Definition \ref{def:snc}), and its independence of the choice of such a model will be deduced from the weak factorization theorem for quasi-excellent schemes in characteristic zero \cite{AbTe} (Theorem \ref{thm:WF}). The use of the weak factorization theorem to prove well-definedness of motivic invariants is certainly not a new idea: see for instance Section 7.8 in \cite{veys-intro}, \cite{GLM} or Appendix A in \cite{NiSh}. The additional complication in our set-up is the role of the canonical form and the fact that we allow birational modifications on $X$, and not only in the special fiber of some snc-model of $X$. The change of variables formula for motivic integrals \cite{DLgerms}, which is typically used to compute motivic invariants on suitable birational modifications (e.g., in Kontsevich's theorem), has not been developed for motivic integrals of canonical forms on $K$-varieties, which is our main reason to invoke weak factorization instead.To manage explicit calculations, we use the formalism of logarithmic geometry as in \cite{BuNi}.

The most important open problem around motivic zeta functions is the monodromy conjecture, which predicts a precise relation between poles of the motivic zeta function and monodromy eigenvalues, and thus expresses a subtle  connection between the geometry  and the cohomology of degenerations. 
  As a first application of our main result, we will prove that this property is invariant under birational equivalence; that is, if a smooth and proper $\C\llpa t\rrpa$-variety with trivial canonical bundle satisfies the monodromy conjecture, then the same holds for all such varieties in the same birational equivalence class (Theorem \ref{thm:moncon}). 

In our second application, we will explain how the motivic zeta function of $X$ serves as an obstruction to the existence of a smooth and proper model of $X$ in the category of algebraic spaces over $\C\llbr t\rrbr$, even up to birational modification of $X$ and finite base change (extracting roots of $t$). We demonstrate this obstruction in Theorem \ref{thm:CvS} on a recent example by Cynk and van Straten of a degenerating family of complex Calabi-Yau threefolds with trivial monodromy that has no smooth filling \cite{CvS}. We further illustrate it in Theorem \ref{thm:lefschetz} in the context of Voisin's result on the non-existence of smooth fillings for Lefschetz degenerations \cite{voisin}. Along the way, we show that the monodromy conjecture is satisfied in both cases.

\subsection*{Acknowledgement} We are grateful to Radu Laza for several discussions related to the content of Section \ref{sec:nofill}. Lunardon was supported by the EPSRC Centre for Doctoral Training ``London School of Geometry and Number Theory''. Nicaise was supported
by EPSRC grant EP/S025839/1, grants G079218N and G0B17121N of the
Fund for Scientific Research--Flanders, and long term structural funding (Methusalem grant) by the Flemish Government. A part of the results presented here were included in Lunardon's PhD thesis \cite{LunPhD} under Nicaise's supervision.

\subsection*{Notations}
We denote by $k$ a field of characteristic zero, and we set $R=k\llbr t\rrbr$ and $K=k\llpa t \rrpa$. We fix an algebraic closure $k^a$ of $k$, and we set 
$$K^a=\bigcup_{k',\,n}k'\llpa t^{1/n}\rrpa$$
where  $k'$ runs over the finite extensions of $k$ in $k^a$, and $n$ over the positive integers.  
 The field $K^a$ is an algebraic closure of $K$. Throughout Section \ref{sec:nofill}, we will assume that $k$ is algebraically closed; then $K^a$ is the field of Puiseux series over $k$.

For every positive integer $n$, we set $K(n)=k\llpa t^{1/n}\rrpa$ and $R(n)=k\llbr t^{1/n}\rrbr$. 
We denote by $\mu_n=\Spec k[\zeta]/(\zeta^n-1)$ the group scheme of $n$-th roots of unity over $k$. We let it act on $\Spec R(n)$ and $\Spec K(n)$ from the left; the action is given on affine coordinate rings by $$t^{1/n}\mapsto \zeta^{-1} t^{1/n}\otimes \zeta.$$
If $k$ contains all $n$-th roots of unity, then $K(n)$ is Galois over $K$ and the above action  identifies $\mu_n(k)$ with the Galois group $\mathrm{Gal}(K(n)/K)$.

  We set
$$\hat{\mu}=\lim_{\longleftarrow}\mu_n$$
where the positive integers $n$ are ordered by divisibility and the transition morphisms
$\mu_{mn}\to \mu_n$ are the $m$-th power maps. Thus $\hat{\mu}$ is the profinite group scheme of roots of unity over $k$. The profinite group $\hat{\mu}(k^a)$ is canonically isomorphic to the intertia subgroup $I_K$ of the absolute Galois group $\mathrm{Gal}(K^a/K)$.
 If $k\subset \C$, then  $\hat{\mu}(\C)$ is isomorphic to the profinite group $\widehat{\Z}$ {\em via} the morphism of topological groups that sends the topological generator $1$ of $\widehat{\Z}$ to the element $(\exp(2\pi i/n))_{n>0}$ of $\hat{\mu}(\C)$.

For every scheme  $\cX$ over $R$, we denote by $\cX_k=\cX\otimes_R k$ its special fiber, and by $\cX_K=\cX\otimes_R K$ its generic fiber.
 If $X$ is a smooth $K$-scheme of pure dimension $d$, then a volume form on $X$ is a nowhere vanishing differential form of degree $d$ on $X$.

\section{Birational invariance of motivic zeta functions}
\subsection{The equivariant Grothendieck ring of varieties}
Let $G$ be a profinite group scheme over $k$. We say that a quotient group scheme $H$ of $G$ is {\em admissible} if $H$ is finite over $k$ and the kernel of
$G(k^a)\to H(k^a)$ is an open subgroup of the profinite group $G(k^a)$.  For our purposes, the most important example is $G=\hat{\mu}$, the profinite group scheme of roots of unity over $k$; its admissible quotients are precisely the finite group schemes $\mu_n$ of $n$-th roots of unity over $k$.

 We now recall the definition of the Grothendieck ring $\gro^{G}(\Var_k)$ of $k$-varieties with $G$-action, which is due to Denef and Loeser (\S2.4 in  \cite{DLbarc}).
 As an abelian group, $\gro^{G}(\Var_k)$ is characterized by the following presentation:
\begin{itemize}
\item {\em Generators}: isomorphism classes of $k$-schemes $X$ of finite type endowed with a good $G$-action. Here ``good'' means that the action factors through an admissible quotient of $G$ and that we can cover $X$ by $G$-stable affine open subschemes (the latter condition is always satisfied when $X$ is quasi-projective). Isomorphism classes are taken with respect to $G$-equivariant isomorphisms.
\item {\em Relations}: we consider two types of relations.
\begin{enumerate}
\item {\em Scissor relations}: if $X$ is a $k$-scheme of finite type with a good $G$-action and $Y$ is a $G$-stable closed subscheme of $X$, then
$$[X]=[Y]+[X\setminus Y].$$
\item \label{it:reltriv} {\em Trivialization of affine actions}: let $X$ be a $k$-scheme of finite type with a good $G$-action, and let $A\to X$ be an affine bundle of rank $d$, endowed with an affine action of $G$ such that $A\to X$ is equivariant (in other words, $A\to X$ is a $G$-equivariant torsor whose translation space is a vector bundle of rank $d$ on $X$ with $G$-linear action). Then $[A]=[X\times_k \A^d_k]$, where the action on $\A^d_k$ is trivial.
\end{enumerate}
\end{itemize}
The group $\gro^{G}(\Var_k)$ has a unique ring structure such that $[X]\cdot [X']=[X\times_k X']$ for all $k$-schemes $X$, $X'$ of finite type  with good $G$-action, where $G$ acts 
on $X\times_k X'$ diagonally. The unit element in this ring is $1=[\Spec k]$.
 We write $\LL$ for the class of $\A^1_k$ (with the trivial $G$-action) in the ring $\gro^{G}(\Var_k)$. The localization $\gro^{G}(\Var_k)[\LL^{-1}]$ is denoted by $\mathcal{M}^{G}_k$.

 When $G$ is the trivial group, we write $\gro(\Var_k)$ and $\mathcal{M}_k$ instead of $\gro^{G}(\Var_k)$ and $\mathcal{M}^G_k$.
 If $G'\to G$ is a continuous morphism of profinite group schemes, then there is an obvious ring morphism
$$\Res^G_{G'}:\gro^{G}(\Var_k)\to \gro^{G'}(\Var_k)$$
obtained by restricting the $G$-action to $G'$. It localizes to a ring morphism
$$\Res^G_{G'}:\mathcal{M}^{G}_k\to \mathcal{M}^{G'}_k.$$

\subsection{The motivic zeta function of a degeneration of varieties with trivial canonical bundle}
Let $X$ be a smooth and proper $K$-scheme of finite type and pure dimension $d$. Assume that the canonical bundle of $X$ is trivial and let $\omega$ be a volume form on $X$. The motivic zeta function $Z_{X,\omega}(T)$ of
 the pair $(X,\omega)$ is an invariant that measures how $X$ and $\omega$ degenerate at $t=0$; we refer to \cite{mallorca} for an introduction.
 It is defined as a generating series of motivic integrals of $\omega$ on $X$ over the totally ramified extensions $K(n)$ of $K$:
\begin{equation}\label{eq:defzeta}
Z_{X,\omega}(T)=\sum_{n>0}\left(\int_{X\otimes_K K(n)}|\omega\otimes_K K(n)|\right)T^n \in \mathcal{M}_k^{\hat{\mu}}\llbr T \rrbr
\end{equation}
See Section 7.2 in \cite{BuNi} for the definition of the motivic integrals in the coefficients of this generating series. The motivic zeta function depends on the choice of the uniformizer $t$ in $K$, unless $k$ contains all roots of unity: then $K(n)$ is the unique totally ramified extension of degree $n$ of $K$, up to $K$-isomorphism.  

The expression \eqref{eq:defzeta}  can be computed explicitly on
 a strict normal crossings model for $X$ over $R$, by Theorem 7.2.1 in \cite{BuNi}. We will recall the formula in Theorem \ref{thm:snc} after setting up some notations.

\begin{definition}\label{def:snc0}
	Let $X$ be a smooth and proper $K$-scheme. An snc-model for $X$ is a proper flat $R$-scheme  $\cX$, endowed with an isomorphism $\cX_K\to X$, such that $\cX$ is regular and its special fiber $\cX_k$ is a strict normal crossings divisor (not necessarily reduced).
\end{definition}

 Such a model always exists, by Nagata's compactification theorem and Hironaka's embedded resolution of singularities for $R$-schemes of finite type. Let $\cX$ be an snc-model for $X$. We write $\cX_k=\sum_{i\in I}N_i E_i$ where $E_i,\,i\in I$ are the prime components of $\cX_k$ and the coefficients $N_i$ are their multiplicities.
 For every non-empty subset $J$ of $I$, we set
$$E_J=\bigcap_{j\in J}E_j,\quad E_J^o=E_J\setminus \left(\bigcup_{i\notin J}E_i\right).$$
The family $\{E_J^o\,|\,\emptyset \neq J\subset I\}$ is a partition of $\cX_k$ into locally closed subsets. We endow $E_J$ and $E_J^o$ with their induced reduced structures. We also set
$N_J=\gcd\{N_j\,|\,j\in J\}$. We denote by $\cX(N_J)$ the normalization of $\cX\otimes_R R(N_J)$, and we set $\widetilde{E}_J^o=E_J^o\times_{\cX}\cX(N_J)$.
 The action of the group scheme $\mu_{N_J}$ on $\Spec R(N_J)$ induces an action on $\cX(N_J)$, which restricts to an action on  $\widetilde{E}_J^o$. This makes $\widetilde{E}_J^o$ into a $\mu_{N_J}$-torsor over $E_J^o$ (see Section 2.3 in \cite{Ni-tameram} for a more explicit description of this torsor).

 We can view $\omega$ as a rational section of the logarithmic relative canonical bundle
 $$\omega_{\cX/R}(\cX_{k,\red}-\cX_k).$$ Then $\omega$ determines a Cartier divisor
 $$\mathrm{div}_{\cX}(\omega)=\sum_{i\in I}\nu_iE_i$$
 supported on the special fiber $\cX_k$. We call $\{(N_i,\nu_i)\,|\,i\in I\}$
 the set of numerical data associated with $(\cX,\omega)$.

\begin{theorem}\label{thm:snc}
We have 
\begin{equation}\label{eq:snc}
Z_{X,\omega}(T)=\sum_{\emptyset \neq J\subset I}[\widetilde{E}_J^o](\LL-1)^{|J|-1}\prod_{j\in J}\frac{\LL^{-\nu_j}T^{N_j}}{1-\LL^{-\nu_j}T^{N_j}}
\end{equation}
in $\mathcal{M}_k^{\hat{\mu}}\llbr T \rrbr$.
\end{theorem}
\begin{proof}
This is Theorem 7.2.1 in \cite{BuNi}.	
\end{proof}

 The motivic zeta function $Z_{X,\omega}(T)$ depends on the volume form $\omega$ only in a mild way: if $X$ is geometrically connected (which is the only case of interest in this article) then any two volume forms differ by a factor $\lambda$ in $K^{\times}$. It follows immediately from the definition of the motivic zeta function (or the explicit formula \eqref{eq:snc}) that
\begin{equation}\label{eq:scalar}
Z_{X,\lambda \omega}(T)=Z_{X,\omega}(\LL^{-\mathrm{ord}_t(\lambda)}T)
\end{equation}
where $\mathrm{ord}_t$ denotes the $t$-adic valuation on $K$.

We will need one further property of the motivic zeta function, namely, its behaviour under extensions of the base field $K$.
 For every element $Z(T)=\sum_{i=0}^{\infty} a_i T^i$ of  $\mathcal{M}_k^{\hat{\mu}}\llbr T \rrbr$, and every positive integer $m$, we denote by $\hat{\mu}(m)$ the kernel of the projection morphism $\hat{\mu}\to \mu_m$, and set
 $$Z^{(m)}(T)=\sum_{i=0}^{\infty}\mathrm{Res}^{\hat{\mu}}_{\hat{\mu}(m)}(a_{mi})T^i \quad \in \mathcal{M}_k^{\hat{\mu}(m)}\llbr T \rrbr.$$

\begin{prop}\label{prop:bc}
For every positive integer $m$, we have 
$$Z_{X\otimes_K K(m),\omega\otimes_K K(m)}(T)=Z^{(m)}_{X,\omega}(T).$$
\end{prop}
\begin{proof}
This follows immediately from the definition of the motivic zeta function in \eqref{eq:defzeta}. \end{proof}

\subsection{Snc-models for germs of canonical forms}
Let $F$ be a finitely generated extension of $K$ of transcendence degree $d$, and let $\omega$ be a non-zero element of $\Omega^d_{F/K}$. Let $X$ be a smooth and proper model for $F$ over $K$, that is, a connected smooth and proper $K$-scheme endowed with a  $K$-isomorphism from $F$ to the function field of $X$. Such a model always exists, by Hironaka's resolution of singularities. We view $\omega$ as a rational section of the canonical line bundle $\Omega^d_{X/K}$ on $X$, and denote the corresponding Cartier divisor on $X$ by $\mathrm{div}_X(\omega)$. 

\begin{definition}
We say that $\omega$ is {\em effective} if $\mathrm{div}_X(\omega)$ is effective for some smooth and proper model $X$ for $F$ over $K$.
\end{definition}

This definition does not depend on the choice of the  model $X$, by the birational invariance of the space $H^0(X,\Omega^d_{X/K})$. 
 From now on, we assume that $\omega$ is effective.
Let $\cX$ be a connected regular proper flat $R$-scheme, endowed with an isomorphism of
 $K$-algebras between $F$ and the function field of $\cX$. Then the generic fiber $\cX_K$ is a smooth and proper model for $F$ over $K$.
 We denote by $D$ the schematic closure in $\cX$ of the effective divisor $\mathrm{div}_{\cX_K}(\omega)$ on $\cX_K$.

\begin{definition}\label{def:snc}
 We say that $\cX$ is an {\em snc-model} for $\omega$ if the divisor $D+\cX_k$ on $\cX$ has strict normal crossings.
\end{definition}

Such an snc-model always exists, by Nagata's compactification theorem and Hironaka's embedded resolution of singularities for $R$-schemes of finite type. A morphism $\cY\to \cX$ of snc-models of $\omega$ is a dominant morphism of $R$-schemes such that the morphism of function fields $F(\cX)\to F(\cY)$ commutes with the given isomorphisms with $F$. In particular, the induced morphism $\cY_K\to \cX_K$ is birational. Note that the support of   $\mathrm{div}_{\cY_K}(\omega)$ is the union of the exceptional divisor of $\cY_K\to \cX_K$ with the strict transform of the support of $\mathrm{div}_{\cX_K}(\omega)$.

\begin{definition}
Let $\cX$ be an snc-model for $\omega$, and denote by $D$ the schematic closure of
$\mathrm{div}_{\cX_K}(\omega)$ in $\cX$.
 An {\em $\omega$-elementary blow-up} of $\cX$
 is a
  blow-up $\cX'\to \cX$ whose center is a connected regular closed subscheme $Z$ of
  $\cX$ that has strict normal crossings with $D+ \cX_k$.
\end{definition}

 The strict normal crossings condition means that, Zariski-locally around every point of $Z$, we can find an effective divisor $D_0$ on $\cX$ such that $D'=D_0+D+\cX_k$ has strict normal crossings and $Z$ is an intersection of irreducible components of $D'$.
  A straightforward local computation shows that the blow-up $\cX'$ is again an snc-model for $\omega$.

 \begin{theorem}\label{thm:WF}
Let $F$ be a finitely generated extension of $K$ of transcendence degree $d$, and let $\omega$ be an effective element of $\Omega^d_{F/K}$.
  Let $\mathcal{I}$ be an invariant of snc-models of $\omega$, that is,
 a map from the set of isomorphism classes of snc-models of $\omega$ to some set $A$. Assume that, for every snc-model $\cX$ of $\omega$ and
  every $\omega$-elementary blow-up $\cX'\to \cX$, we have $\mathcal{I}(\cX)=\mathcal{I}(\cX')$. Then the invariant $\mathcal{I}$ takes the same value on all snc-models of $\omega$.
 \end{theorem}
 \begin{proof}
Let $\cX$ and $\cY$ be two snc-models for $\omega$. Then we can find dense open subschemes $U$ and $V$ of $\cX_K$ and $\cY_K$, respectively, such that $\omega$ does not vanish anywhere on $U$ and $V$, and such that there exists an isomorphism of $K$-schemes $U\to V$ whose associated morphism of function fields $F(V)\to F(U)$ commutes with the given isomorphisms to $F$. Let $D$ be the schematic closure of $\mathrm{div}_X(\omega)$ in $\cX$. In Chapter 1, \S2, Theorem $I_2^{N,n}$ of \cite{HironakaI}, Hironaka established embedded resolution of singularities for the pair $(\cX,\cX\setminus U)$ with respect to the boundary divisor $\cX_k+D$ (note that by the conventions at the start of Chapter 1 in \cite{HironakaI}, the notion of {\em algebraic scheme} includes schemes of finite type over $R$). This implies that we can find a composite $f\colon \cX'\to \cX$ of finitely many $\omega$-elementary blow-ups such that $f$ is an isomorphism over $U$ and the complement of $f^{-1}(U)$ in $\cX'$ (endowed with its induced reduced structure) is a strict normal crossings divisor. The analogous property holds for $\cY$. Thus we may assume that the complement of $U$ in $\cX$ and the complement of $V$ in $\cY$ are strict normal crossings divisors.

 Applying Corollary 5.7.12 in \cite{flattening} to the graph of the birational map $\cX\dashrightarrow \cY$, we can find a blow-up $g\colon \cZ\to \cX$ whose center is disjoint from $U$ and such that the isomorphism $g^{-1}(U)\to V$ extends to a blow-up $g\colon \cZ\to \cY$.
 After a further blow-up of $\cZ$, we may also assume that $\cZ\setminus g^{-1}(U)$ is a strict normal crossings divisor on $\cZ$; in particular, $\cZ$ is an snc-model for $\omega$.
  Now it follows from the weak factorization theorem for admissible blow-ups of quasi-excellent schemes of characteristic zero (Theorem 1.2.1 in \cite{AbTe}) that $\cX$, $\cY$ and $\cZ$ can be connected by a chain of $\omega$-elementary blow-ups and blow-downs between snc-models of $\omega$ (the statement in \cite{AbTe} uses normal crossings boundaries instead of strict normal crossings boundaries, but an inspection of the proof shows that $D_{V_i}$ has strict normal crossings for every $i$ if this holds for $D_1$ and $D_2$). Consequently,  $\mathcal{I}(\cX)=\mathcal{I}(\cY)=\mathcal{I}(\cZ)$.
  \end{proof}

\subsection{Motivic zeta functions of germs of canonical forms}\label{ss:germs}
Let $F$ be a finitely generated extension of $K$ of transcendence degree $d$, and let $\omega$ be an effective element of $\Omega^d_{F/K}$. Let $\cX$ be an snc-model for $\omega$ and let $U$ be the maximal open subscheme of $\cX_K$ where $\omega$ does not vanish. Then $\cX\setminus U$ (with its induced reduced structure) is a strict normal crossings divisor on $\cX$; we denote its prime components by $E_i,\,i\in I$. We write $I_v$ for the set of indices $i\in I$ such that $E_i$ is vertical (that is, contained in $\cX_k$) and $I_h=I\setminus I_v$ for the set of indices $i\in I$ such that $E_i$ is horizontal. For every $i\in I$, we denote by $N_i$ the multiplicity of $\cX_k$ along $E_i$; thus $\cX_k=\sum_{i\in I_v}N_i E_i$ and $N_i=0$ for all $i$ in $I_h$. We view $\omega$ as a rational section of the logarithmic relative canonical bundle $\omega_{\cX/R}(\cX_{k,\red}-\cX_k)$ and we denote by
$\mathrm{div}_{\cX}(\omega)$ the associated divisor on $\cX$. Then we can write
$$\mathrm{div}_{\cX}(\omega)=\sum_{i\in I}\nu_i E_i$$ where $\nu_i>0$ for all $i$ in $I_h$.

 For every non-empty subset $J$ of $I$, we set
$$E_J=\bigcap_{j\in J}E_j,\quad E_J^o=E_J\setminus \left(\bigcup_{i\notin J}E_i\right).$$
The family $\{E_J^o\,|\,\emptyset \neq J\subset I\}$ is a partition of $\cX\setminus U$ into locally closed subsets. We endow $E_J$ and $E_J^o$ with their induced reduced structures.
 If $J$ contains an element of $I_v$, we also set
$N_J=\gcd\{N_j\,|\,j\in J\}$. We denote by $\cX(N_J)$ the normalization of $\cX\otimes_R R(N_J)$, and we set $\widetilde{E}_J^o=E_J^o\times_{\cX}\cX(N_J)$.
 The action of the group scheme $\mu_{N_J}$ on $\Spec R(N_J)$ induces an action on $\cX(N_J)$, which restricts to an action on  $\widetilde{E}_J^o$. This makes $\widetilde{E}_J^o$ into a $\mu_{N_J}$-torsor over $E_J^o$.

 In order to define the motivic zeta function of the pair $(\cX,\omega)$, we need to work in the localization  
 $$\mathcal{R}=\mathcal{M}_k^{\hat{\mu}}\llbr T \rrbr \left[ \frac{1}{1-\LL^{-m}} \right]_{m\in \Z_{>0}}$$
of $\mathcal{M}_k^{\hat{\mu}}\llbr T \rrbr$. 
   There is a unique morphism of $\mathcal{M}^{\hat{\mu}}_k\llbr T \rrbr$-algebras from $\cR$ to  $\widehat{\mathcal{M}}_k^{\hat{\mu}}\llbr T \rrbr$, where $\widehat{\mathcal{M}}_k^{\hat{\mu}}$ denotes the usual dimensional completion of $\mathcal{M}_k^{\hat{\mu}}$  (see Chapter 2, \S4.3 in \cite{CLNS}). The ring $\widehat{\mathcal{M}}_k^{\hat{\mu}}$ is  the standard value ring in the theory of motivic integration, but we prefer to work in $\cR$ to retain finer information: it is not known whether the morphisms $\mathcal{M}_k^{\hat{\mu}}\llbr T\rrbr \to \cR$ and $\cR\to \widehat{\mathcal{M}}_k^{\hat{\mu}}$ are injective.

\begin{definition}\label{def:zeta2}
We define the motivic zeta function of the pair $(\cX,\omega)$ by means of the formula
$$Z_{\cX,\omega}(T)= \sum_{J\subset I,\,J\cap I_v\neq \emptyset}[\widetilde{E}_J^o](\LL-1)^{|J|-1}\prod_{j\in J}\frac{\LL^{-\nu_j}T^{N_j}}{1-\LL^{-\nu_j}T^{N_j}}$$ in $\cR$.
\end{definition}
Note that, for all $i$ in $I$, we have $\nu_i>0$ whenever $N_i=0$, so that the expression in Definition \ref{def:zeta2} indeed defines an element in $\cR$. When $\omega$ is a volume form on $\cX_K$, then $N_i>0$ for all $i$ in $I$, so that $Z_{\cX,\omega}(T)$ is defined already in $\mathcal{M}^{\hat{\mu}}_k\llbr T\rrbr$; then it coincides with the motivic zeta function $Z_{\cX_K,\omega}(T)$, by Theorem \ref{thm:snc}. We will always explicitly indicate whether we consider this invariant in $\mathcal{M}^{\hat{\mu}}_k\llbr T\rrbr$ or in $\cR$.

 \begin{theorem}\label{thm:inv}
  Let $F$ be a finitely generated extension of $K$ of transcendence degree $d$, and let $\omega$ be an effective  element of $\Omega^d_{F/K}$. Let $\cX$ be an snc-model for $\omega$.
 Then the motivic zeta function $Z_{\cX,\omega}(T)$ in $\cR$ only depends on $\omega$, and not on the choice of the snc-model $\cX$.
 \end{theorem}
 \begin{proof}
This type of statement is typically proved by means of a suitable change of variables formula for motivic integrals, but in our set-up (canonical forms on $K$-varieties), such a formula has not been worked out in the literature. Instead, we invoke weak factorization in the form of Theorem \ref{thm:WF}.

 By Theorem \ref{thm:WF}, it suffices to show that $Z_{\cX,\omega}(T)$ remains invariant under $\omega$-elementary blow-ups of $\cX$. Let $Z$ be a connected regular closed subscheme of $\cX$ such that $Z$ has strict normal crossings with the divisor $\sum_{i\in I}E_i$, and let $\cX' \to \cX$ be the blow-up at $Z$. We must prove that $Z_{\cX,\omega}(T)=Z_{\cX',\omega}(T)$.

   For every open subscheme $\cU$ of $\cX$, we set
   $$Z_{\cU,\omega}(T)= \sum_{J\subset I,\,J\cap I_v\neq \emptyset}[\widetilde{E}_J^o\times_{E_J^o}(E_J^o \cap \cU)](\LL-1)^{|J|-1}\prod_{j\in J}\frac{\LL^{-\nu_j}T^{N_j}}{1-\LL^{-\nu_j}T^{N_j}}$$ in $\cR$.
 Then the scissor relations in the Grothendieck ring of varieties immediately imply that, for every finite cover $\{\cU_{\alpha}\,|\,\alpha\in A\}$ of $\cX$ by non-empty open subschemes $\cU_{\alpha}$, we have
 $$Z_{\cX,\omega}(T)=\sum_{\emptyset \neq B\subset A}(-1)^{|B|-1}Z_{\cU_B,\omega}(T),$$ 
 where $\cU_B=\cap_{\beta\in B}\cU_{\beta}$. 
 The analogous property holds for $\cX'$. Therefore, it suffices to prove that we can cover $\cX$ by open subschemes $\cU$ such that
 $Z_{\cV,\omega}(T)=Z_{\cV',\omega}(T)$ for every open subscheme $\cV$ of $\cU$, where $\cV'\to \cV$ is the blow-up at $Z\cap\cV$.
 
 Locally around each point of $Z$, we can find an open subscheme $\cU$ of $\cX$ and a divisor $D_0$
 on $\cU$ such that $D'=D_0+\sum_{i\in I}(E_i\cap \cU)$ has strict normal crossings and $Z\cap \cU$
 is an intersection of irreducible components of $D$. This property then also holds on every open subscheme $\cV$ of $\cU$.
 Let $\cU'\to \cU$ be the blow-up at $Z\cap\cU$. We will show that $Z_{\cU,\omega}(T)=Z_{\cU',\omega}(T)$ by means of the logarithmic techniques in \cite{BuNi}; we refer to Section 3 of \cite{BuNi} for the basic notions from logarithmic geometry that we need.
 
 We endow $\cU$ with the divisorial logarithmic structure induced by the divisor $D'$. Then $\cU$ is log smooth over $\Spec R$ with its standard log structure, by Proposition 3.6.1 in \cite{BuNi}, and $Z\cap \cU$ is the closure of a stratum in the logarithmic stratification of $\cU$.  The blow-up $\cU'\to \cU$ is the log modification associated with a proper subdivision of the fan of $\cU$.
 
 First, assume that $\omega$ does not vanish on $\cU_K$, and that $\cU'_K\to \cU_K$ is an isomorphism.
  By the explicit description of the characteristic monoids in Example 4.1.1 in \cite{BuNi}, our definition of $Z_{\cU,\omega}(T)$ coincides with the right hand side of equation (3) in Theorem 5.3.1 of \cite{BuNi}, applied to the log scheme $\cU$ (see also
 Theorem 7.2.1 in \cite{BuNi} for a discussion of the $\hat{\mu}$-action).
 It then follows from Step 2 of the proof of Theorem 5.3.1 in \cite{BuNi} that $Z_{\cU,\omega}(T)=Z_{\cU',\omega}(T)$.
 
  We will now explain why the assumption that  $\omega$ does not vanish on $\cU_K$ and  $\cU'_K\to \cU_K$ is an isomorphism can be omitted, provided that we work in the localization  $\cR$ of $\mathcal{M}_k^{\hat{\mu}}\llbr T \rrbr$. We apply the formulas (2) and (3) in the statement of Theorem 5.3.1 in \cite{BuNi} to the log scheme $\cU$ and the canonical form induced by $\omega$ on $\cU_K$.
  Expression (2) 
  can be viewed as en element in $\cR$ by the obvious analog of Lemma 5.1.1 in \cite{BuNi}. Consequently, the logarithmic expression in (3) also defines an element in $\cR$, and the explicit description of the characteristic monoids in Example 4.1.1 in \cite{BuNi} again implies that this element is equal to $Z_{\cU,\omega}(T)$. By the same argument as in Step 2 of the proof of Theorem 5.3.1 in \cite{BuNi}, the logarithmic expression for $Z_{\cU,\omega}(T)$  implies that it is invariant under the log modification $\cU'\to \cU$, even without the assumption that $\cU'_K\to \cU_K$ is an isomorphism. Therefore, $Z_{\cU,\omega}(T)=Z_{\cU',\omega}(T)$.
\end{proof}

In view of Theorem \ref{thm:inv}, the following definition is unambiguous.
\begin{definition}\label{def:germ}
Let $F$ be a finitely generated extension of $K$ of transcendence degree $d$, and let $\omega$ be an effective element of $\Omega^d_{F/K}$. We define the motivic zeta function $Z_{\omega}(T)$ by
$$Z_{\omega}(T)=Z_{\cX,\omega}(T) \in \cR$$
where $\cX$ is any snc-model for $\omega$.
\end{definition}

\begin{prop}\label{prop:Ktriv}
Let $X$ be a geometrically connected smooth and proper $K$-scheme with trivial canonical bundle, and let $\omega$ be a volume form on $X$. Then $Z_{\omega}(T)$ is the image of
$Z_{X,\omega}(T)$ in the ring $\cR$.
\end{prop}
\begin{proof}
This follows at once from the formula \eqref{eq:defzeta} for $Z_{X,\omega}(T)$ and the definition of $Z_{\omega}(T)$, because any snc-model of $X$ is also an snc-model of $\omega$.
\end{proof}
\begin{cor}[Birational invariance]\label{cor:birat}
Let $X$ and $X'$ be connected smooth and proper $K$-schemes with trivial canonical bundles, and let  $f\colon X'\dashrightarrow X$ be a birational map. Let $\omega$ be a volume form on $X$, and let $\omega'$ be the unique volume form on $X'$ such that $f^*\omega=\omega'$ on the locus where $f$ is defined. Then
$Z_{X,\omega}(T)=Z_{X',\omega'}(T)$ in the ring $\cR$.
\end{cor}
\begin{proof}
We have  $Z_{X,\omega}(T)=Z_{\omega}(T)=Z_{X',\omega'}(T)$ in $\cR$.
\end{proof}
It is plausible that the equality in Corollary \ref{cor:birat} holds already in $\mathcal{M}_k^{\hat{\mu}}\llbr T \rrbr$, but this result seems out of reach with existing techniques. Note that, already in the setting of Kontsevich's theorem, it is unknown whether birational smooth proper  $k$-varieties with trivial canonical bundles have the same class in $\mathcal{M}_k$.

\section{Birational invariance of the monodromy property}
 The most important open problem about motivic zeta functions is the {\em monodromy conjecture}, which predicts a precise relation between poles of the zeta function and monodromy eigenvalues. It goes back to a similar conjecture of Igusa's on local zeta functions of $p$-adic polynomials (see Section 2.3 in \cite{denef-bourbaki}). The motivic formulation for hypersurface singularities is due to Denef and Loeser, and the analogous problem for motivic zeta functions of $K$-varieties with trivial canonical bundle was stated  by Halle and the second-named author in Question 2.7 in \cite{HaNi}; see also Section 2.3 in \cite{HaNi-CY}. We will give the precise statement in Section \ref{ss:monprop}, after some technical preparations in Sections \ref{ss:eigenv} and \ref{ss:poles} that will also be used for the applications in Section \ref{sec:nofill}.

\subsection{Monodromy eigenvalues}\label{ss:eigenv}
 Let $\sigma$ be an element of the inertia subgroup $I_K\cong \hat{\mu}(k^a)$ of $\mathrm{Gal}(K^a/K)$, and let $\ell$ be a prime number. For every smooth and proper $K$-scheme $Y$  and every integer $m\geq 0$, we denote by
 $$\Phi_{Y,m,\sigma}(u)=\det\left(u\cdot \mathrm{Id}-\sigma \,|\, H^m(Y\times_K K^a,\Q_\ell)\right)$$
 the characteristic polynomial of the action of $\sigma$ on the $\ell$-adic cohomology space $H^m(Y\times_K K^a,\Q_\ell)$.

\begin{prop}\label{prop:indepell}
Let $Y$ be a smooth and proper $K$-scheme.
For every $m\geq 0$, the characteristic polynomial $\Phi_{Y,m,\sigma}(u)$ is a product of cyclotomic polynomials in $\Z[u]$. It is independent of the choice of $\ell$. If we assume that $\sigma$ is a topological generator of $I_K$, then it is also independent of the choice of $\sigma$.
\end{prop}
\begin{proof}
  By the Lefschetz principle, we may assume that $k=\C$. In this case, the inertia group $I_K\cong \hat{\mu}(\C)$ has a canonical topological generator, namely, the element $\tau=(\exp(2\pi i/n))_{n>0}$.
 Berkovich proves in \cite{berk-Z} that there exists, for every $m\geq 0$, a canonical $\Q$-vector space $H^m(Y\times_K K^a,\Q)$ with a quasi-unipotent operator $\Pi$, together with isomorphisms
  $$H^m(Y\times_K K^a,\Q)\otimes_{\Q}\Q_{\ell}\to H^m(Y\times_K K^a,\Q_\ell)$$
  for all primes $\ell$ that identify the actions of $\Pi$ and $\tau$.
  Consequently, for every $m\geq 0$, we have
  $$\Phi_{Y,m,\tau}(u)=\det\left(u\cdot \mathrm{Id}-\Pi \,|\, H^m(Y\times_K K^a,\Q)\right),$$ and this is a product of cyclotomic polynomials in $\Z[u]$ that does not depend on $\ell$. 

 By the continuity of the action of $I_K$ on the $\ell$-adic cohomology of $Y$, the map
 $$\Phi_{Y,m}\colon I_K\to \Q_\ell[u],\ \sigma\mapsto \Phi_{Y,m,\sigma}(u)$$ is continuous for every $m\geq 0$.
  On the dense subset $\{\tau^a\,|\,a\in \Z\}$ of $I_K$, it takes values in the discrete closed subring $\Z[u]$ of $\Q_{\ell}[u]$, and it is independent of $\ell$. It follows that $\Phi_{Y,m}$ is locally constant and independent of $\ell$, and that $\Phi_{Y,m,\sigma}(u)$ is a product of cyclotomic polynomials for every $\sigma\in I_K$.

  Finally, assume that $\sigma$ is a topological generator for $I_K$. We must show that $\Phi_{Y,m,\sigma}(u)=\Phi_{Y,m,\tau}(u)$.
 Since we can approximate $\sigma$ by powers of $\tau$, the fact that $\Phi_{Y,m}$ is locally constant implies that $\Phi_{Y,m,\sigma}(u)=\Phi_{Y,m,\tau^a}(u)$ for some integer $a$. By symmetry, we also have $\Phi_{Y,m,\tau}(u)=\Phi_{Y,m,\sigma^b}(u)$ for some integer $b$. Since $\Phi_{Y,m,\sigma}(u)$ and $\Phi_{Y,m,\tau}(u)$ are products of cyclotomic polynomials, this implies that $\Phi_{Y,m,\sigma}(u)=\Phi_{Y,m,\tau}(u)$.
\end{proof}

\begin{definition}
Let $Y$ be a smooth and proper $K$-scheme. 
\begin{enumerate}
\item For every $m\geq 0$, we set 
$$\Phi_{Y,m}(u)=\Phi_{Y,m,\sigma}(u)$$ for any topological generator $\sigma$ of $I_K$. By Proposition \ref{prop:indepell}, this definition does not depend on the choice of $\sigma$, and $\Phi_{Y,m}(u)$ is a product of cyclotomic polynomials. We call $\Phi_{Y,m}(u)$ the $m$-th characteristic polynomial of $Y$.	
	
\item We say that a complex number is a monodromy eigenvalue of $Y$ if it is a root of the characteristic polynomial $\Phi_{Y,m}(u)$, for some $m\geq 0$.
\end{enumerate}
\end{definition}

A useful tool to compute monodromy eigenvalues of $Y$ is the A'Campo formula for the monodromy zeta function of $Y$, the alternating product of its characteristic polynomials.

\begin{prop}\label{prop:acampo}
Let $Y$ be a smooth and proper $K$-scheme, and let $\cY$ be an snc-model of $Y$, with special fiber $\cY_k=\sum_{i\in I}N_iE_i$. Then 
$$\prod_{m\geq 0}\Phi_{Y,m}(u)^{(-1)^{m+1}}=\prod_{i\in I}(u^{N_i}-1)^{-\chi(E_i^o)}$$
where $\chi(E_i^o)$ is the $\ell$-adic Euler characteristic of $E_i^o$.	
\end{prop}	
\begin{proof}
This is contained in Theorem 2.6.2 in \cite{Ni-tameram}.	
\end{proof}

Beware that, due to cancellations between characteristic polynomials in odd and even degrees, this formula is often not sufficient to determine {\em all}  monodromy eigenvalues. It is sufficient, for instance, when $Y$ only has cohomology in even degrees (e.g., when $Y$ is a $K3$ surface).

 \begin{remark}\label{rem:monodromy}
 Assume that $k=\C$. Let $C$ be a smooth complex curve, let $s$ be a closed point of $C$, and let $t$ be a local parameter in $\mathcal{O}_{C,s}$. Then we can identify the completed local ring $\widehat{\mathcal{O}}_{C,s}$ with $R=\C\llbr t\rrbr$. 
 Let $F$ be the fraction field of  $\mathcal{O}_{C,s}$, and let $Y$ be a smooth and proper $F$-scheme. Then the monodromy eigenvalues of $Y\otimes_F K$ have a more classical topological interpretation:   
 $H^m(Y\times_F K^a,\Q_\ell)$ is canonically isomorphic with the degree $m$ singular cohomology with $\Q_\ell$-coefficients of the nearby fiber of $Y$ at $t=0$, by Deligne's comparison theorem for $\ell$-adic and complex analytic nearby cycles in Expos\'e XIV of  \cite{sga7b}. This isomorphism identifies the classical monodromy action on the nearby cohomology with the action of the canonical topological generator of $I_K=\hat{\mu}(\C)$ on 
  $H^m(Y\times_K K^a,\Q_\ell)$. In particular, the monodromy eigenvalues of $Y$ are precisely the eigenvalues of the monodromy operator on the singular cohomology of the nearby fiber of $Y$ at $t=0$. 
\end{remark} 

\subsection{Poles of motivic zeta functions}\label{ss:poles}
The notion of pole of a motivic zeta function needs to be defined with care, because the coefficient ring  $\mathcal{M}^{\hat{\mu}}_k$ is not a domain. Different proposals have appeared in the literature (see in particular \cite{rove}); we will give a definition that is as basic as possible. In the theory of motivic zeta functions, it is common to make a formal substitution $T=\LL^{-s}$ and view the zeta function as a power series in $\LL^{-s}$, in accordance with the $p$-adic setting where $T=p^{-s}$.

\begin{definition}\label{def:poles}
Let $Z(T)$ be an element of 
$$\mathcal{M}^{\hat{\mu}}_k\left[T,\frac{1}{1-\LL^aT^b}  \right]_{(a,b)\in \Z\times \Z_{>0}}.$$
We say that a subset $P$ of $\Q$ is a set of candidate poles for $Z(\LL^{-s})$ if $Z(T)$ lies in 
$$\mathcal{M}^{\hat{\mu}}_k\left[T,\frac{1}{1-\LL^aT^b}  \right]_{(a,b)\in \Z\times \Z_{>0},\,a/b\in P}.$$
We say that a number $s_0\in \Q$ is a pole of $Z(\LL^{-s})$ if it lies in every set of candidate poles of $Z(\LL^{-s})$.

We analogously define the set of poles of an element of 
$$\mathcal{M}^{\hat{\mu}}_k\left[T,\frac{1}{1-\LL^aT^b},\frac{1}{1-\LL^{-m}}  \right]_{(a,b)\in \Z\times \Z_{>0},\,m\in \Z_{>0}}.$$
\end{definition}

\begin{remark}
{\em A priori}, the set of poles of $Z(\LL^{-s})$ could shrink by inverting the elements $1-\LL^{-m}$, because it is not known whether these are zero-divisors. We will always indicate explicitly in which ring we consider the object $Z(T)$.
\end{remark}

\begin{example}\label{exam:poles} Let $X$ be a geometrically connected smooth and proper $K$-scheme with trivial canonical bundle, and let $\omega$ be a volume form on $X$. Let $\cX$ be an snc-model for $X$, and let $$\{(N_i,\nu_i)\,|\,i\in I\}$$ be the set of numerical data associated with $(\cX,\omega)$. Then 
	$$\{-\nu_i/N_i\,|\,i\in I\}$$ is a set of candidate poles for $Z_{X,\omega}(\LL^{-s})$, by the explicit formula in Theorem \ref{thm:snc}. 

		Formula \eqref{eq:scalar} implies that multiplying $\omega$ with $\lambda\in K^{\ast}$ shifts the poles of $Z_{X,\omega}(\LL^{-s})$ by $-\mathrm{ord}_t \lambda$. Moreover, if we denote by $P$ the set of poles of $Z_{X,\omega}(\LL^{-s})$, then it follows from Proposition \ref{prop:bc} that, for every integer $n>0$, the set of poles of $Z_{X\otimes_K K(n),\omega \otimes_K K(n)}(\LL^{-s})$ is contained in 
		$\{ns_0\,|\,s_0\in P\}$, but this inclusion is often strict: there are examples of $K3$ surfaces $X$ whose motivic zeta function has multiple poles (see for instance Example 5.3.5 in \cite{HaNi-CY}), but when $n$ is sufficiently divisible, then $X\otimes_K K(n)$ has a Kulikov model in the category of algebraic spaces over $R(n)$ and this implies that $Z_{X\otimes_K K(n),\omega \otimes_K K(n)}(\LL^{-s})$ has a unique pole (see \cite{StVo} and \cite{HaNi-CY}). 
\end{example}

Typically, many of the candidate poles coming from an snc-model are not actual poles of $Z_{X,\omega}(\LL^{-s})$. For one thing, the set of candidates depends heavily on the choice of the snc-model; but even the intersection of these sets over all snc-models is often much larger than the set of poles. One can create smaller sets of candidate poles by replacing snc-models by log-smooth models as in Theorem 5.3.1 of \cite{BuNi}, but even these sets can still be too large to compute the set of poles. On the other end, the numerical data of a minimal dlt-model of $X$ (see \cite{KNX}) only capture the largest pole of $Z_{X,\omega}(\LL^{-s})$. There is no clear candidate for a type of $R$-model of $X$ whose numerical data compute exactly the poles of $Z_{X,\omega}(\LL^{-s})$.

Definition \ref{def:poles} seems difficult to use in practice because we do not control all sets of candidate poles. In order to prove that a set $P\subset \Q$ is the set of poles of $Z(\LL^{-s})$, one can try the following strategy.

\medskip{\em Step 1: show that $P$ is a set of candidate poles by finding a suitable expression for $Z(T)$.} {\em A priori}, there is no guarantee that the set of poles is also a set of candidate poles, but this is the case in all known examples of motivic zeta functions.

\medskip{\em Step 2: show that the set of poles is not strictly contained in $P$.} For this purpose, one can apply a motivic realization morphism, that is, a ring morphism from $\mathcal{M}^{\hat{\mu}}_k$ to a more concrete ring (ideally, a domain). Many examples can be found in Chapter 2 of \cite{CLNS}. For the applications in Section \ref{sec:nofill}, we will use the Euler-Poincar\'e realization 
$$\mathrm{EP}\colon \mathcal{M}^{\hat{\mu}}_k\to \Z[w,w^{-1}],\, [S]\mapsto \mathrm{EP}(S)$$
described in Proposition 3.5.10 of Chapter 2 in \cite{CLNS}. Since $k$ has characteristic zero, this ring morphism is uniquely characterized by the property that it sends the class of a smooth and proper $k$-scheme $S$ with good $\hat{\mu}$-action to the Euler-Poincar\'e polynomial 
$$\mathrm{EP}(S)=\sum_{m\geq 0}(-1)^mb_m(S)w^m$$
where $b_m(S)$ denotes the $m$-th $\ell$-adic Betti number of $S$ (the invariant $\mathrm{EP}(S)$ ignores the $\hat{\mu}$-action on $S$). When $S$ is not smooth and proper, one can compute $\mathrm{EP}(S)$ from the weight filtration on the $\ell$-adic cohomology of $S$ with compact supports,  or by writing $[S]$ in terms of classes of smooth and proper $k$-schemes in $\Gro^{\hat{\mu}}(\Var_k)$. 
For instance, $\mathrm{EP}(\mathbb{P}^1_k)=1+w^2$ and, by the scissor relations,  $\mathrm{EP}(\LL)=\mathrm{EP}(\Pro^1_k)-\mathrm{EP}(\Spec k)=w^2$. Applying $\mathrm{EP}$ to the coefficients of $Z(T)$, we obtain an element $Z^{\mathrm{EP}}(T)$ of $\Q(w,T)$. To prove that an element $s_0\in P$ is a pole of $Z(\LL^{-s})$, we can perform a direct residue calculation to check whether there exists a complex root of unity $\zeta$ such that $\zeta w^{-2s_0}$ is a pole of $Z^{\mathrm{EP}}(T)$. Of course, this is a sufficient, but not a necessary condition, because the realization morphism $\mathrm{EP}$ has a large kernel. Depending on the series $Z(T)$, one may need to find finer realization morphisms to prove that the elements of $P$ are actual poles. Note that $\mathrm{EP}$ also detects poles of $Z(T)$ viewed as an element in 
$\cR$, because $\mathrm{EP}$ localizes to a ring morphism 
$$\mathcal{M}_k^{\hat{\mu}}\left[ \frac{1}{1-\LL^{-m}}\right]_{m>0}\to \Q(w).$$

\begin{example}\label{exam:largestpole}
	Let $X$ be a geometrically connected smooth and proper $K$-scheme with trivial canonical bundle and let $\omega$ be a volume form on $X$. Let $$\{(N_i,\nu_i)\,|\,i\in I\}$$ be the set of numerical data associated with $(\cX,\omega)$, for some  snc-model $\cX$ for $X$, and let $s_{\max}$ be the maximum of the associated set of candidate poles 
	$$\{-\nu_i/N_i\,|\,i\in I\}.$$ 
 Using a residue calculation on the Euler-Poincar\'e realization $Z^{\mathrm{EP}}(T)$ as in Step 2 above, it is proved in Theorem 3.2.3 of \cite{HaNi-CY} that $s_{\max}$ is always a pole of $Z_{X,\omega}(\LL^{-s})$; the proof shows that this remains true when we view $Z_{X,\omega}(\LL^{-s})$ as an element in $\cR$. In particular, $s_{\max}$ is independent of the choice of the snc-model $\cX$. 
\end{example}
 
\subsection{The monodromy property}\label{ss:monprop}
The following property is the analog for varieties with trivial canonical bundle of Igusa's monodromy conjecture for $p$-adic local zeta functions (Conjecture 2.3.2 in \cite{denef-bourbaki}).

\begin{definition}\label{def:MP}
 Let $X$ be a geometrically connected smooth and proper $K$-scheme with trivial canonical bundle. 
    Let $\mathcal{E}_X$ be the set of couples $(a,b)$ in $\Z^2$ such that $b>0$ and $\exp(2\pi i a/b)$
  is a monodromy eigenvalue of $X$. We say that $X$ satisfies the {\em monodromy property} if, for any volume form $\omega$ on $X$, the motivic zeta function  $Z_{X,\omega}(T)$ belongs to the ring
 $$\mathrm{Mon}_X=\mathcal{M}^{\hat{\mu}}_k\left[T,\frac{1}{1-\LL^a T^b} \right]_{(a,b)\in \mathcal{E}_X}.$$
\end{definition}
This property does not depend on $\omega$, by equation \eqref{eq:scalar}.  It implies that, for every pole $s_0$ of  $Z_{X,\omega}(\LL^{-s})$, the value $\exp(2\pi i s_0)$ is a monodromy eigenvalue of $X$. 

\begin{example}\label{exam:largestpolemon}
With the notations from Example \ref{exam:largestpole}, it is proved in Theorem 3.3.3 in \cite{HaNi-CY} that $\exp(2\pi i s_{\max})$ is always a monodromy eigenvalue on the cohomology of $X$ of middle degree $\dim(X)$.
\end{example}

 \begin{question}\label{ques:moncon}
 Does every geometrically connected smooth and proper $K$-scheme with trivial canonical bundle satisfy the monodromy property? 
 \end{question}

Based on the analogy with Igusa's monodromy conjecture, we believe that the answer to Question \ref{ques:moncon} should be positive, but only a limited number of cases are known. The question is trivial if $X$ has a snc-model with reduced special fiber, because then all the poles of $Z_{X,\omega}(\LL^{-s})$ are integers and all monodromy eigenvalues of $X$ are equal to $1$. If $X$ satisfies the monodromy property, then so does $X\otimes_K K(n)$ for any $n>0$ by Example \ref{exam:poles}, but the converse implication is unknown (and equivalent to Question \ref{ques:moncon} by the semistable reduction theorem). It is not hard to prove that 
the class of varieties that satisfy the monodromy property is closed under products (Proposition 7.2 in \cite{pagano}).

It was shown in \cite{HaNi} that abelian varieties satisfy the monodromy property. This result was generalized in \cite{HaNi-CY} to varieties $X$ with a so-called Galois-equivariant Kulikov model; this includes the case of Kummer surfaces, by Theorem 6.2 in \cite{Overkamp}. In all these cases, it is proved that $Z_{X,\omega}(\LL^{-s})$ has a unique pole, so that the result follows from Example \ref{exam:largestpolemon}.

A positive answer is also known for triple-point-free degenerations of $K3$-surfaces \cite{JasPhD,LunPhD} and for $K3$-surfaces that have a regular proper $R$-model such that the reduced special fiber has at worst $ADE$ singularities \cite{LunPhD}. It was proved in \cite{pagano} that, when $S$ is an abelian surface over $K$, or a $K3$-surface that satisfies the monodromy property,  the Hilbert schemes $\mathrm{Hilb}^n(S)$ also satisfy the monodromy property for all $n>0$. All these cases contain examples where 
$Z_{X,\omega}(\LL^{-s})$ has multiple poles. A basic case that is still open are $K3$-surfaces with finite monodromy action on their $\ell$-adic cohomology spaces.
 
 The motivic zeta function $Z_{X,\omega}(T)$ has a partial cohomological interpretation in the form of a trace formula (see Theorem 5.4 in \cite{NiSe} and Theorem 6.4 in \cite{Ni-trace}): for every positive integer $n$, the Euler characteristic specialization of the degree $n$ coefficient in  $Z_{X,\omega}(T)$ equals the $n$-th Lefschetz number of the monodromy on $X$ (that is, the alternating sum of the traces of $\sigma^n$ on the $\ell$-adic cohomology spaces of $X$, for any topological generator $\sigma$ of $I_K$). However, this Euler characteristic specialization destroys all information about the exponents $\nu_i$ in the formula \eqref{eq:defzeta}, so that it cannot say much about the poles of $Z_{X,\omega}(T)$.

\subsection{Birational invariance of the monodromy conjecture}
As an application of Theorem \ref{thm:inv}, we will show that the answer to Question \ref{ques:moncon} only depends on the birational equivalence class of $X$, up to inverting the elements $1-\LL^{-m}$ in $\mathrm{Mon}_X$ for all $m>0$.

\begin{prop}\label{prop:eigen}
Let $X$ and $X'$ be geometrically connected smooth and proper $K$-schemes with trivial canonical bundles, and assume that $X$ and $X'$ are birational.
 Then $\Phi_{X,m}(u)=\Phi_{X',m}(u)$ for all $m\geq 0$. In particular, $X$ and $X'$ have the same monodromy eigenvalues.
\end{prop}
\begin{proof}
By Kontsevich's theorem, the schemes $X$ and $X'$ have the same class in the completed localized Grothendieck ring of $K$-varieties $\widehat{\mathcal{M}}_K$ (see Chapter 7, \S3.5.5 in \cite{CLNS}).
By Chapter 2, \S4.3.9 in \cite{CLNS}, this implies that for all $m\geq 0$ and all primes $\ell$, the $\mathrm{Gal}(K^a/K)$-modules  $H^m(X\times_K K^a,\Q_\ell)$ and $H^m(X'\times_K K^a,\Q_\ell)$ have the same class in the Grothendieck ring of $\ell$-adic $\mathrm{Gal}(K^a/K)$-representations. In particular, $\Phi_{X,m}(u)=\Phi_{X',m}(u)$.
\end{proof}

\begin{theorem}\label{thm:moncon}
Let $X$ and $X'$ be geometrically connected smooth and proper $K$-schemes with trivial canonical bundle. Assume that $X$ and $X'$ are birational. Let $\omega$ and $\omega'$ be volume forms on $X$ and $X'$, respectively.  Then the image of   $Z_{X,\omega}(T)$ in $\cR$ lies in
  $\mathrm{Mon}_X[(1-\LL^{-m})^{-1}]_{m>0}$ if and only if the analogous property holds for $(X',\omega')$.
\end{theorem}
\begin{proof} Since the desired property is independent of $\omega$ and $\omega'$, 
 this follows immediately from Corollary \ref{cor:birat} and Proposition \ref{prop:eigen}.
\end{proof}

In other words, up to inverting the elements $1-\LL^{-m}$, the variety $X$ satisfies the monodromy property if and only if this is true for $X'$. This extends the list of known cases given after Question \ref{ques:moncon}. In particular, we obtain the following result.

\begin{cor}\label{cor:biratHilb}
	Let $S$ be an abelian surface over $K$, or a $K3$ surface that satisfies the monodromy property, and let $n$ be a positive integer. 
Let $X$ be a geometrically connected smooth and proper $K$-scheme with trivial canonical bundle. Assume that $X$ is birational to $\mathrm{Hilb}^n(S)$, for some $n>0$. Then $X$ satisfies the monodromy property up to inverting $1-\LL^{-m}$ for all $m>0$: for every volume form $\omega$ on $X$, the image of   $Z_{X,\omega}(T)$ in $\cR$ lies in
$\mathrm{Mon}_X[(1-\LL^{-m})^{-1}]_{m>0}$.
\end{cor}
\begin{proof}
This follows at once from Theorem \ref{thm:moncon} and the main result of \cite{pagano}.	
\end{proof}	

\section{Degenerations with trivial monodromy and no smooth fillings}\label{sec:nofill}
 Throughout this section, we assume that $k$ is algebraically closed.
 
 \subsection{Monodromy and good reduction}
Let $X$ be a geometrically connected smooth and proper $K$-scheme. It is natural to ask when we can extend $X$ to a smooth and proper family $\cX$ over the formal disk $\Spec R$. This property is formalized in the following definition.

\begin{definition}\label{def:goodred}
We say that $X$ has good reduction if there exists a  smooth and proper algebraic space $\cX$ over $R$ such that $\cX\times_R K$ is isomorphic to $X$.
\end{definition}

Allowing $\cX$ to be an algebraic space offers more flexibility: from Proposition 5.1 in \cite{matsumoto}, one can deduce the existence of a $K3$-surface $X$ over $K$ that has good reduction in the sense of Definition \ref{def:goodred}, but does not have a smooth and proper $R$-model in the category of schemes, even after base change to any finite extension of $K$. A strong obstruction to  good reduction is provided by the monodromy action on the cohomology of $X$.  
 \begin{definition}\label{def:cohgood}
 	We say that $X$ has cohomological good reduction if the monodromy action of $I_K$ on the $\ell$-adic cohomology spaces $H^m(X\times_K K^a,\Q_\ell)$ is trivial, for all primes $\ell$ and all $m\geq 0$.
 \end{definition}
 
 The smooth and proper base change theorem implies that, if $X$ has good reduction, then it has cohomological good reduction. Thus, the triviality of the monodromy action on the cohomology is a necessary condition for $X$ to have good reduction. However, it is not always sufficient; let us recall a few classical results.

\begin{enumerate}
\item {\em Curves.} If $X$ is a curve of genus zero over $K$, then $X$ is isomorphic to $\mathbb{P}^1_K$ by the triviality of the Brauer group of $K$. In particular, it has good reduction.
If $X$ is a genus one curve, then $I_K$ acts trivially on $H^1(X\times_K K^a,\Q_\ell)$ if and only if the Jacobian of $X$ has good reduction; in that case, $X$ has good reduction if and only if it has a rational point (which implies that it is isomorphic to its Jacobian, an elliptic curve). For curves $X$ of genus $g\geq 2$, the triviality of the action of $I_K$ is again equivalent to good reduction of the Jacobian. This is also equivalent to the property that the minimal snc-model of $X$ has reduced special fiber and the dual graph of this special fiber is a tree. To characterize good reduction of the curve $X$, one needs a finer criterion in terms of the \'etale fundamental group \cite{oda}.

\item {\em Abelian varieties.} If $X$ is an abelian variety, then Serre and Tate have proved in \cite{serre-tate} that $X$ has good reduction if and only if $I_K$ acts trivially on $H^1(X\times_K K^a,\Q_\ell)$. This is called the {\em N\'eron-Ogg-Shafarevich} criterion. In this case, $X$ has a canonical smooth and proper $R$-model in the category of schemes: the N\'eron model of $X$ is an abelian scheme over $R$.

\item {\em $K3$-surfaces.} For $K3$-surfaces, good reduction corresponds to the Type I case in the Kulikov classification of semistable degenerations of $K3$-surfaces. It follows from the appendix in \cite{HT} that a $K3$-surface $X$ over $K$ has good reduction if and only if $I_K$ acts trivially on $H^2(X\times_K K^a,\Q_\ell)$.

\item {\em Hyperk\"ahler varieties.} The result for $K3$-surfaces has been partially generalized to  hyperk\"ahler varieties in \cite{KLSV}, in the complex analytic setting. Let $X$ be a projective family of hyperk\"ahler manifolds over a punctured complex disk $\Delta^*=\Delta\setminus \{0\}$, meromorphic at the origin $0\in \Delta$. Assume that the monodromy action on $H^2(X_t,\Q)$ is trivial, where $t$ is a point of $\Delta^*$. Then, after a base change to a finite cover of $\Delta^*$,
 there exists a projective family of hyperk\"ahler manifolds $Y$ that is bimeromorphic to $X$ and that extends to a smooth projective family over $\Delta$. 
\end{enumerate}

Cases 2--4 raise the question what can be said for other classes of varieties with trivial canonical bundle. One could hope that the hyperk\"ahler case generalizes as follows: when $X$ has trivial canonical  bundle and cohomological good reduction, then there exist a finite extension $K'$ of $K$ and a smooth and proper $K'$-scheme $Y$ with trivial canonical bundle such that $Y$ is birational to $X\otimes_K K'$ and $Y$ has good reduction. Unfortunately, this is not always the case. Counterexamples were given by Voisin   \cite{voisin}, Wang \cite{wang} and Cynk--van Straten \cite{CvS}. 
  A positive answer exists if we allow mild singularities on our model for $X$.

\begin{theorem}\label{thm:MMP}
Let $X$ be a geometrically connected smooth projective $K$-scheme of dimension $d$ with trivial canonical bundle. Assume that $I_K$ acts trivially on the middle $\ell$-adic cohomology space	
	$H^d(X\times_K K^a,\Q_\ell)$. Then the special fiber of any minimal dlt-model for $X$ over $R$ is irreducible.
\end{theorem}
\begin{proof}
 For the definition and the existence of a minimal dlt-model of $X$, we refer to \cite{KNX}.
 We can reduce to the case $k=\C$.
 It is explained in Section 2 of \cite{StVo} how,  for every $m\geq 0$, one can attach to $X$ a $\Q$-rational  mixed Hodge structure $$(H^m_{\mathrm{lim}}(X),F^{\bullet},W_{\bullet})$$ 
 whose underlying $\Q$-vector space $H^m_{\mathrm{lim}}(X)$ is canonically isomorphic with the space $H^m(X\times_K K^a,\Q)$ defined in \cite{berk-Z}.   
  When $X$ is defined over a complex curve as in Remark \ref{rem:monodromy}, then this limit mixed Hodge structure coincides with the degree $m$ limit mixed Hodge structure at $t=0$. The weight filtratrion $W_{\bullet}$ is the monodromy weight filtration associated with the monodromy operator $\Pi$ on $H^m(X\times_K K^a,\Q)$.
  
  In particular, if $I_K$ acts trivially on $H^m(X\times_K K^a,\Q_{\ell})$ for some prime $\ell$, then $H^m_{\mathrm{lim}}(X)$ is pure of weight $m$. By Theorem 3.3.3 in \cite{HaNi-CY}, this  implies that the so-called {\em essential skeleton} of $X$ is a point. The essential skeleton is homeomorphic to the dual complex of the special fiber of any minimal dlt-model $\cX$ of $X$ over $R$, by Theorem 24 in \cite{KNX}, so that $\cX_k$ is irreducible.
\end{proof}

 Let $X$ be a geometrically connected smooth and proper $K$-scheme with trivial canonical bundle and  cohomological good reduction. Even if this does not imply good reduction in general, 
the monodromy property in Question \ref{ques:moncon} suggests that the geometry of snc-models of $X$  should still be quite special: since $I_K$ acts trivially on the $\ell$-adic cohomology of $X$, all the monodromy eigenvalues of $X$ are equal to $1$, so that we expect that, for any volume form $\omega$ on $X$, the motivic zeta function $Z_{X,\omega}(T)$ has only integer poles; more precisely, it should lie in
$$\mathcal{M}^{\hat{\mu}}_k\left[T,\frac{1}{1-\LL^a T^b} \right]_{a\in \Z,\,b\in \Z_{>0},\,a/b\in \Z}.$$
This also imposes further restrictions on the singularities that may appear on a minimal dlt-model for $X$. In the following sections, we will prove that the monodromy property is satisfied in the examples by Cynk--van Straten (Theorem \ref{thm:CvS}) and Voisin  (Theorem \ref{thm:lefschetz}), where we have trivial monodromy but no good reduction. We will also explain how the motivic zeta function acts as an obstruction to good reduction in these examples.

We do not know if  the monodromy property  holds in all cases where the monodromy action is trivial; this is an interesting test case for Question \ref{ques:moncon}. It is open even for hyperk\"ahler varieties, because, even if we assume cohomological good reduction, the construction of a smooth model in \cite{KLSV} still requires base change to a finite extension of $K$, which destroys information about the motivic zeta function. To the best of our knowledge, it is unknown whether hyperk\"ahler varieties with cohomological good reduction always have good reduction, as in the case of $K3$ surfaces.

\subsection{Cynk and van Straten's example}

\begin{lemma}\label{lemm:obstr}
Let $X$ be a geometrically connected smooth and proper $K$-scheme with trivial canonical bundle, and let $\omega$ be a volume form on $X$. Assume that there exists a  smooth and proper $K$-scheme $Y$ with trivial canonical bundle such that $Y$ is birational to $X$ and has good reduction. Then there exist an integer $\nu$ and an element $\alpha$ in the image of $$\mathrm{Res}^{\{1\}}_{\widehat{\mu}}\colon \mathcal{M}_k\to \mathcal{M}^{\hat{\mu}}_k$$ such that
$$Z_{X,\omega}(T)=\frac{\alpha}{1-\LL^{-\nu}T}$$
in $\cR$. In particular, $-\nu$ is the only pole of $Z_{X,\omega}(\LL^{-s})$ (viewed as an object in $\cR|_{T=\LL^{-s}}$).  
 \end{lemma}
 \begin{proof}
 	By the birational invariance of the motivic zeta function (Corollary \ref{cor:birat}) we may assume that $X$ itself has good reduction. Let $\cX$ be a smooth and proper algebraic space over $R$ such that $\cX_K$ is isomorphic to $X$. The explicit formula for the motivic zeta function in \eqref{eq:defzeta} is also valid in the category of algebraic spaces, by Section 7 of \cite{HaNi-CY}. This immediately implies the result.
 \end{proof}

\begin{theorem}\label{thm:CvS}
There exists a geometrically connected smooth and proper $K$-scheme $X$ of dimension $3$ with trivial canonical bundle that has the following properties:
\begin{enumerate}
\item \label{it:montriv} the inertia group $I_K$ acts trivially on the $\ell$-adic cohomology spaces $H^m(X\times_K K^a,\Q_\ell)$, for all $m\geq 0$ and all primes $\ell$;
\item \label{it:nogoodred} there is no  smooth and proper scheme $X'$ over a finite extension $K'$ of $K$ such that $X'$ has trivial canonical bundle, $X'$ is birational to $X\otimes_K K'$ and $X'$ has good reduction over the valuation ring in $K'$;
\item \label{it:twopoles} for any volume form $\omega$ on $X$, the motivic zeta function $Z_{X,\omega}(\LL^{-s})$ has precisely two poles, and these are integers (see the proof for a more precise statement). In particular, $X$ satisfies the monodromy property.
\end{enumerate}
\end{theorem}
\begin{proof}
We will use the beautiful example from \cite{CvS}. Let $\Q^a$ be the algebraic closure of $\Q$ in $\C$ and fix an embedding of $\Q^a$ into $k$. Cynk and van Straten start from a singular complex Calabi-Yau threefold that arises as a double cover of $\PP^3_{\Q^a}$ ramified along a particular arrangement of $8$ planes. By carefully degenerating the arrangement and applying an explicit birational modification to the resulting family of double covers, they obtain a projective flat family 
 $f\colon \cY\to \A^1_{\Q^a}=\Spec \Q^a[u]$	with the following properties:
 \begin{itemize}
 	\item the scheme $\cY$ is regular along $\cY_0=f^{-1}(0)$;
 	\item the fiber of $\cY$ over a general point in $\A^1_{\Q^a}(\C)$ is a smooth Calabi-Yau threefold with Hodge numbers $h^{1,1}=41$ and $h^{1,2}=1$;
 	\item the special fiber $\cY_0$ is reduced, the singular locus of $\cY_0$ is a line $L$, and $\cY_0$ is double along $L$ with exactly $4$ pinch points;
 	\item the blow-up of $\cY_0$ along $L$ is a smooth Calabi-Yau threefold $Z$ with Hodge numbers   $h^{1,1}=46$ and $h^{1,2}=0$; 
 	\item the monodromy action on the cohomology of the  complex nearby fiber of $\cY$ at $u=0$ has order $2$.
 	 \end{itemize}
 	 We consider $R=k\llbr t\rrbr$ as a $\Q^a[u]$-algebra {\em via} the morphism of $\Q^a$-algebras that sends $u$ to $t^2$, and we set $\cX=\cY\otimes_{\Q^a[u]}R$.
Then $X=\cX_K$ satisfies \eqref{it:montriv}, by Remark \ref{rem:monodromy}. Although this is not explicitly stated in \cite{CvS}, it also satisfies \eqref{it:nogoodred}: suppose for a contradiction that such $K'$ and $X'$ exist, and let $\cX'$ be a smooth and proper algebraic space over the valuation ring $R'$ in $K'$ such that $\cX'_{K'}$ is isomorphic to $X'$. Then $\cX'_k$ is a smooth and proper algebraic space over $k$ with trivial canonical bundle. It has the same $\ell$-adic Euler characteristic as $X'$ by smooth and proper base change for $\ell$-adic cohomology. Batyrev's theorem on invariance of Betti numbers under $K$-equivalence \cite{batyrev} implies that $X$ and $X'$ have the same $\ell$-adic Euler characteristic. 

On the other hand, $\cX\times_R R'$ is still normal because $\cX_k$ is reduced, so that $\cX'_k$ is birational to $\cX_k\cong \cY_0\otimes_{\Q^a}k$, and therefore also to $Z$. Batyrev's theorem again implies that $\cX'_k$ has the same $\ell$-adic Euler characteristic as $Z$, contradicting the fact that $Z$ and the smooth fibers of $\cY$ have different Euler characteristics (Batyrev's theorem is only stated for algebraic varieties, but the same proof applies to algebraic spaces). 
 
 We will now show that $X$ satisfies \eqref{it:twopoles} and deduce another (but ultimately related) argument for \eqref{it:nogoodred}.
   To obtain a more complete picture, we will first compute the motivic zeta function of 
  $W=\cY\otimes_{\Q^a[u]}k\llpa u\rrpa$, and then extract a square root $t$ of $u$ to recover the variety $X=W\otimes_{k\llpa u\rrpa}K$.

 Section 3.2 in \cite{CvS} describes how to construct an snc-model for $W$ over $k\llbr u\rrbr$. Consider the blow-up $\cY'\to \cY$ at the singular line $L$ in the special fiber $\cY_0$, and set $\cW=\cY'\otimes_{\Q^a[u]}k\llbr u\rrbr$. Then $\cW$ is an snc-model for $W$. The special fiber $\cW_k$ has two irreducible components: the strict transform $E_0=Z$ of $\cY_0$, and the exceptional divisor $E_1$ of the blow-up. Since $\cY_0$ is double along $L$, we have $\cW_k=E_0+2E_1$, so that $N_0=1$ and $N_1=2$. Since $\cW_k$ has multiplicity $1$ along $E_1$, we can choose a volume $\omega_0$ on $W$ that extends to a relative volume form on $\cW\setminus E_1$ over $k\llbr u\rrbr$; equivalently, $\nu_0=0$. Since $\cY$ is regular along $\cY_0$, we have $K_{\cY'/\cY}=2E_1$, so that $\nu_1=1$ (recall that we use the logarithmic twist $\omega_{\cW/k\llbr u\rrbr}(\cW_{k,\red}-\cW_k)$ as reference line bundle in the definition of the invariants $\nu_i$). By the explicit formula in Theorem \ref{thm:snc},  
 $$Z_{W,\omega_0}(T)=[E_0^o]\frac{T}{1-T}+[\widetilde{E}_1^o]\frac{\LL^{-1}T^2}{1-\LL^{-1}T^2}+[E_{0,1}](\LL-1)\frac{\LL^{-1}T^3}{(1-T)(1-\LL^{-1}T^2)}$$ in $\mathcal{M}_k^{\hat{\mu}}\llbr T \rrbr$.
  This gives $s=0$ and $s=-1/2$ as candidate poles for $Z_{W,\omega_0}(\LL^{-s})$ and confirms that $W$ satisfies the monodromy property in Definition \ref{def:MP}: the values $1$ and $-1$ are monodromy eigenvalues of $W$, since the monodromy action on the cohomology of the  complex nearby fiber of $\cY$ at $0$ has order $2$ and it is trivial in degree $0$. In fact, using the A'Campo formula for the monodromy zeta function of $W$ (Proposition \ref{prop:acampo}), one easily computes that $1$ and $-1$ are monodromy eigenvalues in degree $3$. 
 	
 Denote by $\omega=\omega_0\otimes_{k\llpa u \rrpa}K$ the volume form induced by $\omega_0$ on $X$ {\em via} base change to $K$. A standard calculation shows that the normalization of $\cW\otimes_{k\llbr u\rrbr}R$ is an snc-model for $X$ with special fiber
 $$\cX_k=F_0+F_1$$
 where $F_0$ is isomorphic to $Z$, $F_{0,1}$ is isomorphic to $E_{0,1}$, and $F_1^o$ is isomorphic to $\widetilde{E}^o_1$.
  Using either this snc-model
 or Proposition \ref{prop:bc} to compute $Z_{X,\omega}(T)$, we find  $$Z_{X,\omega}(T)=[F_0^o]\frac{T}{1-T}+[F_1^o]\frac{\LL^{-1}T}{1-\LL^{-1}T}+[F_{0,1}](\LL-1)\frac{\LL^{-1}T^2}{(1-T)(1-\LL^{-1}T)}$$
  so that $\{-1,0\}$ is a set of candidate poles for  $Z_{X,\omega}(\LL^{-s})$. 
  
  We will prove that these are actual poles by means of the Euler-Poincar\'e realization $\mathrm{EP}$ as in  Section \ref{ss:poles}. Applying $\mathrm{EP}$ to the coefficients of $Z_{X,\omega}(T)$, we obtain the rational function 
$$\mathrm{EP}(F^o_0)\frac{T}{1-T}+\mathrm{EP}(F^o_1)\frac{w^{-2}T}{1-w^{-2}T}+\mathrm{EP}(F_{0,1})(w^2-1)\frac{w^{-2}T^2}{(1-T)(1-w^{-2}T)}$$
in $\Q(w,T)$.
 By Example \ref{exam:largestpolemon}, the largest candidate pole $0$ is an actual pole of the Euler-Poincar\'e realization of $Z_{X,\omega}(\LL^{-s})$.  The Betti numbers of $F_1$ and $F_{0,1}$ are computed in Section 3.2 of \cite{CvS}; the only thing we need is that $F_1$ has non-trivial cohomology in odd degree while $F_{0,1}$ does not, so that $\mathrm{EP}(F_1^o)=\mathrm{EP}(F_1)-\mathrm{EP}(F_{0,1})$ has a term in odd degree whereas  $\mathrm{EP}(F_{0,1})$ only has even degree terms. A direct residue calculation now shows that $w^{2}=\mathrm{EP}(\LL)$ is a pole of the Euler-Poincar\'e realization of $Z_{X,\omega}(T)$. It follows that $\{-1,0\}$ is the set of poles of $Z_{X,\omega}(\LL^{-s})$, and also of its image in $\cR|_{T=\LL^{-s}}$.

By Lemma \ref{lemm:obstr}, this implies that $X$ is not birational to a smooth and proper $K$-scheme with trivial canonical bundle and good reduction. Using Proposition \ref{prop:bc}, one checks in the same way that $Z_{X\otimes_K K(m),\omega\otimes_K K(m)}(\LL^{-s})$ has two poles for any $m>0$, which gives an alternative argument for \eqref{it:nogoodred} in the statement.
\end{proof}
\begin{remark}
	With the notations of the proof, $\cY\otimes_{\Q^a[u]}k\llbr u\rrbr$ is a minimal dlt-model for its generic fiber $W$ with integral special fiber, in accordance with Theorem \ref{thm:MMP}. 	
\end{remark}

\subsection{Voisin's result on Lefschetz degenerations without smooth fillings}
 Let $\Delta$ be an open disk around $0$ in the complex plane $\C$, and denote by $t$ the  coordinate function on $\C$. Let $d\geq 4$ be an even integer. In \cite{voisin}, Voisin considers a family $\mathscr{H}\subset \mathbb{P}^{d+1}_{\Delta}$  of projective hypersurfaces of degree $d+2$  such that the total space $\mathscr{H}$ is smooth, the morphism $\mathscr{H}\to \Delta$ is smooth away from $0\in \Delta$,  and the fiber $\mathscr{H}_0$  has a unique singularity $z$, which is an isolated ordinary double point: locally around $z$, there exists a holomorphic coordinate system $(x_0,\ldots,x_{d})$ on $\mathscr{H}$ such that the projection morphism to $\Delta$ is given by $t=x_0^2+\ldots+x_{d}^2$. Then a general fiber of $\mathscr{H}$ is a smooth projective Calabi-Yau variety, and the monodromy transformation on its  cohomology has order $2$ by the Picard-Lefschetz formula. Voisin proves that, even after extracting an arbitrary root of $t$, the family $\mathscr{H}\setminus \mathscr{H}_0$ over $\Delta\setminus \{0\}$ cannot be extended to a smooth and proper morphism $\mathscr{H}'\to \Delta$ such that the fiber over $0$ is cohomologically K\"ahler. We will now use the motivic zeta function to prove a generalized algebraic version of this statement (Theorem \ref{thm:lefschetz}); as a first step, we will explicitly compute the motivic zeta function and deduce that it satisfies the monodromy property.

 Until further notice, we allow $d$ to be any positive integer. For our calculations, we need a standard result on deformations of ordinary double points (see for instance Proposition III.2.8 in \cite{Freitag-Kiehl}). 
 Let $\cX$ be a flat $R$-scheme of finite type of pure relative dimension $d$ and with smooth generic fiber $\cX_K$. Let $z$ be an isolated ordinary double point of the special fiber $\cX_k$. Then there exists an element $a$ in the maximal ideal of $R$ such that the henselization of the local ring $\mathcal{O}_{\cX,z}$ is isomorphic as an $R$-algebra to the henselization of the local ring $$ \left(R[x_0,\ldots,x_d]/(a+x_0^2+\ldots+x_d^2)\right)_{(t,x_0,\ldots,x_d)}.$$ In particular, the completed local ring  $\widehat{\mathcal{O}}_{\cX,z}$ is $R$-isomorphic
to \begin{equation}\label{eq:formodp} 
	R\llbr x_0,\ldots,x_d\rrbr/(a+x_0^2+\ldots+x_d^2).
\end{equation} 
Our assumption that $\cX$ has smooth generic fiber implies that $a\neq 0$.
We call the $t$-adic valuation of $a$ the {\em modulus} of $\cX$ at $z$; it depends only on $\cX$ and $z$, and not on the choice of $a$. 

 We denote by $\Q_d$ the class in $\mathcal{M}^{\hat{\mu}}_k$ of a smooth projective quadric of dimension $d$ with trivial $\hat{\mu}$-action. We also set $\Q_0=2[\Spec k]$. We further denote by $\Q'_d$ the class in $\mathcal{M}^{\hat{\mu}}_k$ of a double cover of $\mathbb{P}^d_k$ ramified along a smooth quadric hypersurface, where $\mu_2$ acts through the involution on the double cover. These definitions are independent of the choice of the various quadrics, by projective equivalence of smooth quadrics. 

\begin{theorem}\label{thm:odp}
Let $\cX$ be a proper flat $R$-scheme such that:
\begin{itemize}
\item  the generic fiber $X=\cX_K$ is smooth and geometrically connected, of dimension $d>0$, and has trivial canonical bundle; 
\item the special fiber $\cX_k$ has finitely many singular points $z_1,\ldots,z_r$ and each of these singularities is an ordinary double point on $\cX_k$.
\end{itemize}
We denote by $m_i$ the modulus of $\cX$ at the singularity $z_i$, and by $\cX_{k,\mathrm{sm}}$ the $k$-smooth locus of $\cX_k$.
 Let $\omega$ be a volume form on $X$ that extends to a relative volume form on $\cX\setminus \{z_1,\ldots,z_r\}$ over $R$. 
   
 For every positive integer $m$, we set
 \begin{equation}\label{eq:odpeven}
 \begin{aligned}	
 Z_m(T)	=& (\Q_d-\Q_{d-1})\frac{\LL^{m(1-d)/2}T}{1-\LL^{m(1-d)/2}T}\\ &  +(\LL-1)\Q_{d-1}\frac{\LL^{m(1-d)/2}T^2+\sum_{j=1}^{(m/2)-1}\LL^{(1-d)j}T }{(1-T)(1- \LL^{m(1-d)/2}T) }  
 \end{aligned}
 \end{equation}
 when $m$ is even, and 

\begin{equation}\label{eq:odpodd}
	\begin{aligned}	
		Z_m(T)	=& (\Q'_d-\Q_{d-1})\frac{\LL^{m(1-d)}T^2}{1-\LL^{m(1-d)}T^2} \\ & +\frac{(\LL-1)\Q_{d-1}}{(1-T)(1- \LL^{m(1-d) }T^2) }\Bigl(\LL^{m(1-d)} T^3\\ &+  \sum_{j=1}^{(m-1)/2}\LL^{(1-d)j}T+ \sum_{j=(m+1)/2}^{m-1}\LL^{(1-d)j}T^2\Bigr)
	\end{aligned}
\end{equation}
 when $m$ is odd. Then the motivic zeta function of $(X,\omega)$ is given by
  
\begin{equation}\label{eq:odpzeta}
Z_{X,\omega}(T)= [\cX_{k,\mathrm{sm}}]\frac{T}{1-T}
+\sum_{i=1}^r Z_{m_i}(T)
\end{equation}
in $\mathcal{M}_k^{\hat{\mu}}\llbr T \rrbr$. 
 The poles of $Z_{X,\omega}(\LL^{-s})$ are $0$ and $m_i(1-d)/2$, for $i=1,\ldots,r$, except when $d=2$: then we must discard the values $m_i(1-d)/2$ where  $m_i$ is even.  The set of poles remains unchanged when we replace $Z_{X,\omega}(T)$ by its image in $\cR$.
 Moreover, the variety $X$ satisfies the monodromy property.
\end{theorem}
\begin{proof}
Let $\cV$ be a separated flat $R$-scheme of finite type, and set $V=\cV_K$.  We assume that $V$ is smooth and has trivial canonical bundle. Let  $\theta$ be a volume form on $V$ and let $W$ be a subscheme of $\cV_k$. Then one can define a motivic zeta function 
$Z_{\cV,W,\theta}(T)$ in $\mathcal{M}^{\hat{\mu}}_k\llbr T\rrbr$ that measures how the pair $(V,\theta)$ degenerates along $W$. It is given by the generating series
\begin{equation}\label{eq:zetasupp}
Z_{\cV,W,\theta}(T)=\sum_{n>0}\left(\int_{]W[\otimes_K K(n)} |\theta\otimes_K K(n)| \right)T^n	
\end{equation}		
where $]W[$ denotes the punctured tube of $\cV$ along $W$ in the category of rigid analytic $K$-varieties; that is, $]W[$ is the generic fiber of the formal completion of $\cV$ along $W$. We refer to \cite{Ni-trace} for the definition of the motivic integrals in the coefficients, and to \cite{hartmann-motint} for a discussion of the $\hat{\mu}$-action. For our purposes, we only need the following properties.

\begin{enumerate}
\item The zeta function $Z_{\cV,W,\theta}(T)$ only depends on $]W[$ and the restriction of $\theta$ to $]W[$. In particular, it is invariant under the following operations: 
\begin{itemize}
	\item replacing $(\cV,W,\theta)$ by another triple $(\cV',W',\theta')$ such that there exists an isomorphism of formal $R$-schemes $\widehat{\cV/W}\to \widehat{\cV'/W'}$ that identifies the restrictions of $\theta$ and $\theta'$ to the generic fibers; 
	\item proper birational modifications of $\cV$ that are isomorphisms on the generic fiber. 
\end{itemize}
 It is also invariant under multiplying $\theta$ with an invertible regular function on $\cV$. 

\item The zeta function $Z_{\cV,W,\theta}(T)$ is additive in $W$: if $\{W_1,\ldots,W_r\}$ is a partition of $W$ into subschemes, then 
$$Z_{\cV,W,\theta}(T)=Z_{\cV,W_1,\theta}(T)+\ldots+Z_{\cV,W_r,\theta}(T).$$

\item  The analog of Theorem \ref{thm:snc} is satisfied:
if $h\colon \cV'\to \cV$ is a proper morphism such that $h_K\colon \cV'_K\to \cV_K$ is an isomorphism, $\cV'$ is regular and $R$-flat, and $\cV'_k=\sum_{i\in I}N_iE_i$ is a strict normal crossings divisor, then 
\begin{equation}\label{eq:sncsupp}
	Z_{\cV,W,\theta}(T)=\sum_{\emptyset \neq J\subset I}[\widetilde{E}_J^o\cap h^{-1}(W)](\LL-1)^{|J|-1}\prod_{j\in J}\frac{\LL^{-\nu_j}T^{N_j}}{1-\LL^{-\nu_j}T^{N_j}}
\end{equation}
where the values $\nu_i$ and the schemes $\widetilde{E}^o_J$ are defined as in Theorem \ref{thm:snc}.

\item The analog of Proposition \ref{prop:bc} is satisfied: for every positive integer $m$, we have 
$$Z_{\cV\otimes_R R(m),W,\theta\otimes_K K(m)}(T)=Z^{(m)}_{\cV,W,\theta}(T).$$

\item If $\cV$ is proper over $R$, then 
$Z_{\cV,\cV_k,\theta}(T)=Z_{\cV_K,\theta}(T)$.		
\end{enumerate}

We start from the regular scheme 
  $$\cV=\Spec R[x_0,\ldots,x_d]/(t+x_0^2+\ldots+x_d^2).$$
We denote by $O$ the origin of $\cV_k$. Let $\theta$ be a relative volume form on $\cV\setminus \{O\}$ over $R$. 
  Blowing up $\cV$ at the point $O$, we obtain a regular scheme $\cV'$ such that $\cV'_k$ has strict normal crossings. We denote by $E_0$ the strict transform of $\cV_k$ in $\cV'$; this scheme is irreducible, except when $d=1$.
     The exceptional divisor $E_1$ of the blow-up is isomorphic to $\Pro^{d}_k$ and intersects $E_0$ along a smooth quadric.   Every irreducible component of $E_0$ has multiplicity $1$ in $\cV'_k$, and $E_1$ has multiplicity $2$. The relative canonical divisor of the blow-up is  $ K_{\cV'/\cV}=dE_1$.
  
  For every positive integer $n$, we denote by $\cV'(n)$ the normalization of $\cV'\otimes_R R(n)$. Let $\widetilde{E}_1$ be the inverse image  of $E_1$ in $\cV'(2)$. This is is a double cover of $E_1\cong \Pro^d_k$ ramified along the smooth quadric $E_0\cap E_1$; it is a smooth $d$-dimensional quadric and $\mu_2$ acts through the involution on the double cover.  
  Plugging this information into the formula \eqref{eq:sncsupp}, we find 
  \begin{equation}\label{eq:odpzeta0}
  Z_{\cV,O,\theta}(T)=(\Q'_d-\Q_{d-1})\frac{\LL^{1-d}T^2}{1-\LL^{1-d}T^2}+(\LL-1)\Q_{d-1}\frac{\LL^{1-d}T^3}{(1-T)(1-\LL^{1-d}T^2)}.
  \end{equation}
	
 By the formal description of ordinary double points in \eqref{eq:formodp}, we can find for each singular point $z_i$ of $\cX_k$ a commutative diagram of rings 
		\begin{center}
	\begin{tikzcd}
	\widehat{\mathcal{O}}_{\cV\otimes_R R(m_i),O} \arrow[r, "\sim"] &	
	\widehat{\mathcal{O}}_{\cX,z_i}
		\\ R(m_i) \arrow[r,"\sim"] \arrow[u] &  \arrow[u] R
	\end{tikzcd}
	\end{center} 
where the upper horizontal arrow is an isomorphism and the lower horizontal arrow is an isomorphism of $k$-algebras that sends $\sqrt[m_i]{t}$ to a uniformizer in $R$. Consequently, we can write 
\begin{eqnarray*}
Z_{X,\omega}(T)&=& Z_{\cX,\cX_k,\omega}(T)
\\ &=& Z_{\cX, \cX_{k,\mathrm{sm}},\omega}(T)+\sum_{i=1}^r Z_{\cX,z_i,\omega}(T)
\\ &=& Z_{\cX, \cX_{k,\mathrm{sm}},\omega}(T)+\sum_{i=1}^r Z^{(m_i)}_{\cV,O,\theta}(T).
\end{eqnarray*}
Since $\omega$ extends to a relative volume form on the $R$-smooth locus of $\cX$, we have
$$Z_{\cX,\cX_{k,\mathrm{sm}},\omega}(T)=[\cX_{k,\mathrm{sm}}]\frac{T}{1-T}.$$
 Using formula \eqref{eq:odpzeta0} for $Z(\cV,O,\theta)(T)$, we can also compute each term $Z^{(m_i)}(\cV,O,\theta)(T)$. 
   Working out the combinatorics (generating series of lattice points in an appropriate $2$-dimensional rational polyhedral cone), one obtains the expression for $Z_{m_i}(T)$ in \eqref{eq:odpeven} or \eqref{eq:odpodd}, depending on the parity of $m_i$.

The candidate poles of expression \eqref{eq:odpzeta} are $0$ and $m_i(1-d)/2$, for $i=1,\ldots,r$. It follows from Example \ref{exam:largestpole} that the largest candidate pole, $0$, is a pole of $Z_{X,\omega}(\LL^{-s})$. Next, we show that $m_i(1-d)/2$ is not a pole of $Z_{X,\omega}(\LL^{-s})$ if $d=2$ and $m_i$ is even. Since $\Q_2=(\LL+1)^2$ and $\Q_1=\LL+1$, we can rewrite $Z_{m_i}(T)$ as 
\begin{equation*}\label{eq:smallres}
\begin{array}{lcl}
Z_{m_i}(T)&=& \displaystyle \LL(\LL+1)\frac{\LL^{-m_i/2}T}{1-\LL^{-m_i/2}T}\\[2ex] & &  \displaystyle +(\LL^2-1)\frac{\LL^{-m_i/2}T^2+\sum_{j=1}^{(m_i/2)-1}\LL^{-j}T }{(1-T)(1- \LL^{-m_i/2}T) }
\\[2ex] &=& \displaystyle (\LL+1)\frac{\LL^{1-m_i/2}T(1-T) +(\LL-1)\LL^{-m+i/2}T^2+(1-\LL^{1-m_i/2})T }{(1-T)(1- \LL^{-m_i/2}T) } 
\\[2ex] &=& \displaystyle \frac{(\LL+1)T}{(1-T)}.
\end{array}
\end{equation*}
Another way to obtain this formula is to use the existence of a small resolution of singularities for the germ $(\cX,z_i)$ in the category of algebraic spaces, and compute  $Z_{\cX,z_i,\omega}(T)$ on this algebraic space; the coefficient $\LL+1$ is the class in $\mathcal{M}_k^{\hat{\mu}}$ of the exceptional curve.

 If $d\neq 2$ or $m_i$ is odd, one can again use the Euler-Poincar\'e realization from Section \ref{ss:poles} to check that $m_i(1-d)/2$ is a pole of $Z_{\cX,\omega}(\LL^{-s})$. We leave the computation as an exercise for the reader. Our argument shows that the set of poles remains unchanged when we replace $Z_{X,\omega}(T)$ by its image in $\cR$.

 The value $1$ is a monodromy eigenvalue on the degree $0$ cohomology of $X$ (and also on the degree $d$ cohomology, by Theorem 3.3.3 in \cite{HaNi-CY}).
 If $d$ is odd or $r=0$, then all the poles of $Z_{X,\omega}(\LL^{-s})$ are integers, so that $X$ satisfies the monodromy property. If $d$ is even and $r\geq 1$, then all the poles other than $0$ are half-integers, and $-1$ is also a monodromy eigenvalue on the degree $d$ cohomology of  $X$ by the \'etale Picard-Lefschetz formula (Theorem 3.4 of Expos\'e XV in \cite{sga7b}). Consequently, the monodromy property is satisfied in this case, as well.
\end{proof}

\begin{corollary}\label{cor:odp}
Let $Y$ be a geometrically connected smooth proper $K$-scheme with trivial canonical bundle. Assume that $Y$ is birational to the generic fiber of an $R$-scheme $\cX$ as in the statement of Theorem \ref{thm:odp}. Then $Y$ satisfies the monodromy property up to inverting the elements $1-\LL^{-m}$ in $\mathcal{M}^{\hat{\mu}}_k\llbr T \rrbr$ for all positive integers $m$.
\end{corollary}
\begin{proof} This follows immediately from Theorems \ref{thm:moncon} and \ref{thm:odp}.
\end{proof}
		
We are now ready to prove an algebraic version of Voisin's result.		
		
\begin{theorem}\label{thm:lefschetz}	
Let $\cX$ be a proper flat $R$-scheme such that:
\begin{itemize}
	\item  the generic fiber $X=\cX_K$ is smooth and geometrically connected, of even dimension $d\geq 4$, and has trivial canonical bundle; 
	\item the special fiber $\cX_k$ has finitely many singular points $z_1,\ldots,z_r$, with $r\geq 1$, and each of these singularities is an ordinary double point on $\cX_k$.
\end{itemize}
Then the monodromy action on the $\ell$-adic cohomology of $X$ has order $1$ or $2$, but
there is no  smooth and proper scheme $X'$ over a finite extension $K'$ of $K$ such that $X'$ has trivial canonical bundle, $X'$ is birational to $X\otimes_K K'$ and $X'$ has good reduction over the valuation ring in $K'$. 
\end{theorem}	
\begin{proof}	
The 
\'etale version of the Picard-Lefschetz formula (Theorem 3.4 in Expos\'e XV of \cite{sga7b}) implies that the action of $I_K$ on the $\ell$-adic cohomology of $X$ is trivial if the modulus of $\cX$ at each singularity $z_i$ is even, and has order $2$ otherwise. 

The remainder of the statement is invariant under finite extensions of $K$, so that it suffices to prove the result for $K'=K$.
 Let $\omega$ be a volume form on $X$. By the birational invariance of the motivic zeta function (Corollary \ref{cor:birat}) and Lemma \ref{lemm:obstr}, it suffices to show that the motivic zeta function 
$Z_{X,\omega}(\LL^{-s})$ has at least two poles when we view it as an object in $\cR$. This follows from Theorem \ref{thm:odp}, because $d\neq 2$ and $r\geq 1$.
\end{proof}

\end{document}